\documentclass[12pt]{article}
\usepackage{amsmath}
\usepackage{amssymb}
\usepackage{amsthm}
\usepackage[mathscr]{eucal}
\theoremstyle{definition}
\newtheorem{definition}{Definition}
\newtheorem{theorem}[definition]{Theorem}
\newtheorem{proposition}[definition]{Proposition}
\newtheorem{lemma}[definition]{Lemma}

\theoremstyle{remark}
\newtheorem{remark}[definition]{Remark}
\newtheorem{example}[definition]{Example}
\newcounter{enumctr}

\newcommand{\R}{\mathbb{R}}
\newcommand{\C}{\mathbb{C}}
\newcommand{\id}{\hbox{id}}
\newcommand{\rT}{\mathrm {T}}
\begin{document}
\title{\vspace*{-10mm}
A Perron-type theorem for fractional linear differential
systems}
\author{
N.D.~Cong\footnote{\tt ndcong@math.ac.vn, \rm Institute of Mathematics, Vietnam Academy of Science and Technology, 18 Hoang Quoc Viet, 10307 Ha Noi, Viet Nam},
T.S.~Doan\footnote{\tt dtson@math.ac.vn, \rm Institute of Mathematics, Vietnam Academy of Science and Technology, 18 Hoang Quoc Viet, 10307 Ha Noi, Viet Nam and Department of Mathematics, Hokkaido University, Japan}
\;and\;
H.T.~Tuan\footnote{\tt httuan@math.ac.vn, \rm Institute of Mathematics, Vietnam Academy of Science and Technology, 18 Hoang Quoc Viet, 10307 Ha Noi, Viet Nam.}
}
\maketitle

\begin{abstract}
We give a necessary and sufficient condition for a system of linear inhomogeneous fractional differential equations to have at least one bounded solution. We also obtain an explicit description for the set of all bounded (or decay) solutions for these systems.
\end{abstract}


\section{Introduction}
In recent years, fractional differential equations (FDEs) have attracted increasing interest due to their varied applications on various fields of science and engineering. Their applications ranging from physics (see, e.g.,  Hilfer~\cite{Hilfer}), image processing (see Bai and Feng~\cite{Bai}), biomechanics (see Freed and Diethelm~\cite{Freed}) to finance  (see Scalas~\cite{Scalas}), and to social sciences  (see Ahmed and El-Khazali~\cite{Ahmed}). For more details on theory of fractional differential equations and its applications, we refer the reader to the monographs  Samko \textit{et al.}~\cite{Samko} and  Diethelm~\cite{Kai} and the references therein. Like the classical theory of ordinary differential equations, the investigation of long term behavior of solutions of fractional differential equations is of  fundamental importance for the theory of fractional differential equations.
Although several results on asymptotic behavior of fractional differential equations are already published (e.g., on stability theory \cite{Ahmed_1, Tisdell2012, Li_Chen_Podlubny2010, Cong_3}, boundedness of solutions \cite{Gallegos}, Lyapunov exponents \cite{Cong_1}, attractivity \cite{Chen_2012}, stable manifolds \cite{Cong_2},\dots), it is surprising to see that much of the basic qualitative foundational theory is yet to be fully developed. One of the reasons for this is the fact that the solution to a fractional differential equation does not generate a semi-group and thus do not generate flow in the classical sense. 

Consider the inhomogenneous system of the order $\alpha \in (0,1)$ involving {\em Caputo derivative}
\begin{equation}\label{mainEq}
^{C\!}D^\alpha_{0+}x(t)=Ax(t)+f(t),
\end{equation}
where $t\in [0,\infty)$, $x(t)\in\R^d$, $A\in\R^{d\times d}$ and $f:[0,\infty)\rightarrow \R^d$. Motivated by Perron's work \cite{Perron}, an interesting question arises here: what conditions must $A$ satisfy in order that \eqref{mainEq} has at least one bounded solution for every continuous vector-valued functions $f(\cdot)$. 
In the case of ordinary differential equations, i.e. in case of \eqref{mainEq} with $\alpha=1$, the answer is known:  $0$ does not belong to the spectrum $\sigma(A)$ of $A$ (see Coppel~\cite[Proposition 3, p. 22]{Coppel}). However, for the fractional case $0<\alpha<1$, the question is still open.

In 1996, Matignon~\cite{Matignon} studied the fractional system \eqref{mainEq}. By using the Laplace transform and the corresponding characteristic equation, he gave a criterion for the external stability of 
\eqref{mainEq}, see \cite[Theorem 4, p. 967]{Matignon}. Since then, many authors have investigated and derived results on stability and convergence of solutions of linear fractional differential systems, see e.g., Deng~\textit{et el.}~\cite{Deng_2007}, Sabatier~\textit{et al.}~\cite{Sabatier}, Mesbahi and Haeri~\cite{Mesbahi}, Abusaksaka and Partington~\cite{Abusaksaka}, and Duarte-Mermoud~\cite{Duarte}. 

In this paper, we consider the fractional system \eqref{mainEq} of the fractional order $\alpha \in (0,1)$ with the external force $f(\cdot)$ in the space of bounded continuous vector-valued functions. We will give a necessary and sufficient condition for this system to have at least one bounded solution. Our main result is
a Perron-type theorem for FDEs (Theorem~\ref{thm.main}) saying that the inhomogeneous system \eqref{mainEq} has at least one bounded solution for every bounded continuous external force $f(\cdot)$ if and only if the matrix $A$ satisfies a hyperbolic condition 
\begin{equation}\label{SpecCond}
\sigma(A) \subset \Lambda_{\alpha}^u \cup \Lambda_{\alpha}^s,
\end{equation}
where $\Lambda_{\alpha}^u, \Lambda_{\alpha}^s$ are defined by \eqref{SectorUnstable}--\eqref{SectorStable}.

Furthermore, we also obtain an explicit description for the set of all bounded solutions of \eqref{mainEq}, and in the case $\lim_{t\to\infty} f(t)=0$ for the set of solutions decaying to $0$. To do this, our approach is as follows. First, we transform the matrix $A$ into its Jordan normal form to obtain a simpler system than the system \eqref{mainEq}. Next, using the variation of constants formula and a procedure of substitution to describe solutions explicitly. Finally, by estimating Mittag-Leffler functions in domains of the complex plane we show asymptotic behavior of solutions which enable us to describe the set of bounded solutions of \eqref{mainEq}. 

The paper is organized as follows. In Section~\ref{sec.preliminaries}, we present some  basics of fractional calculus and some preliminary results related to Mittag-Leffler functions. In Section~\ref{sec.main} we  describe the set of bounded and decaying solutions of \eqref{mainEq} (Theorem \ref{Main Result1}) in the case the matrix $A$ satisfying the hyperbolicity condition \eqref{SpecCond}. Furthermore, by showing that the hyperbolicity condition \eqref{SpecCond} is necessary for \eqref{mainEq} to have at least one bounded solution for any give external force we derive the main result of the paper, Theorem~\ref{thm.main}. 

\section{Preliminaries}\label{sec.preliminaries}
In this section we briefly recall some basics of fractional calculus.
First we introduce notations which are used throughout this paper. 
For a nonzero complex number $\lambda$, we define its argument to be in the interval $-\pi < \arg{(\lambda)}\leq \pi$. For $\alpha \in (0,1)$, we define the sets
\begin{align}
\Lambda_{\alpha}^u &:=\left\{\lambda \in\C\setminus\{0\}:|\arg{(\lambda)}|< \frac{\alpha \pi}{2}\right\},\label{SectorUnstable}\\
\Lambda_{\alpha}^s &:=\left\{\lambda\in\C\setminus\{0\}:|\arg{(\lambda)}|> \frac{\alpha \pi}{2}\right\}. \label{SectorStable}
\end{align}

Let $\R_{\geq 0}$ denote the set of all nonnegative real numbers. For a Banach space $(X,\|\cdot\|)$, we denote by $\left(C_\infty(\R_{\geq 0};X),\|\cdot\|_\infty\right)$ the space of all continuous functions $\xi:\R_{\geq 0}\rightarrow X$ such that
\[
\|\xi\|_\infty:=\sup_{t\in \R_{\geq 0}}\|\xi(t)\|<\infty,
\]
and by $\left(C_\infty^0(\R_{\geq 0};X),\|\cdot\|_\infty\right)$  the space of functions $\xi \in C_\infty(\R_{\geq 0};X)$ satisfying
\[
\lim_{t\to \infty}\|\xi(t)\|=0.
\]
Clearly, $C_\infty^0(\R_{\geq 0};X) \subset C_\infty(\R_{\geq 0};X)$ and both of them are Banach spaces with the norm $\|\cdot\|_\infty$.

Let $\alpha>0$, $[a,b]\subset \R$ and $x:[a,b]\rightarrow \R$ be a measurable function such that $\int_a^b|x(\tau)|\;d\tau<\infty$. Then, the Riemann--Liouville integral operator of order $\alpha$ is defined by
\[
(I_{a+}^{\alpha}x)(t):=\frac{1}{\Gamma(\alpha)}\int_a^t(t-\tau)^{\alpha-1}x(\tau)\;d\tau\quad \hbox{ for } t>a,
\]
where the Gamma function  $\Gamma:(0,\infty)\rightarrow \R$ is defined as
\[
\Gamma(\alpha):=\int_0^\infty \tau^{\alpha-1}\exp(-\tau)\;d\tau.
\]
 The corresponding Riemann-Liouville fractional derivative is given by
\[
(D_{a+}^\alpha) x(t):=(D^mI_{a+}^{m-\alpha}x)(t),
\]
where $D=\frac{d}{dx}$ is the usual derivative and $m:=\lceil\alpha\rceil$ is the smallest integer larger or equal $\alpha$. The \emph{Caputo fractional derivative} $^{C\!}D_{a+}^\alpha x$ of a function $x\in C^m([a,b])$ (see e.g.,\ \cite{Kai}), is defined by
\[
(^{C\!}D_{a+}^\alpha x)(t):=(I_{a+}^{m-\alpha}D^mx)(t),\qquad \hbox{ for } t>a.
\]
The Caputo fractional derivative of a $d$-dimensional vector function $x(t)=(x_1(t),\cdots,x_d(t))^{\rT}$ is defined component-wise as
\[
(^{C\!}D_{a+}^\alpha x)(t):=(^{C\!}D_{a+}^\alpha x_1(t),\cdots,^{C\!}D_{a+}^\alpha x_d(t))^{\rT}.
\]
Throughout this paper, we only consider $\alpha \in (0,1)$.

Let us look at our equation \eqref{mainEq}: if $f(\cdot)$ vanishes then \eqref{mainEq} can be solved explicitly with the help of the Mittag-Leffler functions, which play a fundamental role in investigation of fractional differential equations like the exponential functions do for the ordinary differential equations. The {\em Mittag-Leffler function} is defined for $z\in\C$ as
\[
E_{\alpha,\beta}(z)=\sum_{k=0}^\infty\frac{z^k}{\Gamma(\alpha k+\beta)},\qquad E_\alpha(z):=E_{\alpha,1}(z).
\]
We may substitute $z$ by $A\in\R^{d\times d}$ to get Mittag-Leffler function of matrix variable. In case $f(\cdot)$ vanishes the general solution of \eqref{mainEq} is $E_\alpha(t^\alpha A) x_0$ with $x_0\in\R^d$.

Now, in the case $f(\cdot)$ does not vanish, then, in general, \eqref{mainEq} cannot be solved explicitly. Fortunately, like the ordinary differential equations we may use the so called variations of constants formula to investigate the solutions of the inhomogeneous equation \eqref{mainEq}, namely consider the initial value problem associated with \eqref{mainEq}:
\begin{equation}\label{mainEq1}
\begin{split}
^{C\!}D^\alpha_{0+}x(t)&=Ax(t)+f(t),\\
x(0)&=x_0\in\R^d.
\end{split}
\end{equation}
By using the Laplace transform, we get a representation for solutions of this system as follows.
\begin{theorem}[Variation of constants formula for fractional differential equations]\label{Var_Const_Form}
For any $x_0 \in \R$ the system \eqref{mainEq1} has a unique solution, which we denote 
by $\varphi(\cdot;0,x_0)$. Moreover, the solution $\varphi(\cdot;0,x_0)$ is given by the formula
\[
\varphi(t;0,x_0)=E_\alpha(t^\alpha A)\;x_0+\int_0^t (t-\tau)^{\alpha-1}E_{\alpha,\alpha}((t-\tau)^\alpha A)f(\tau)\;d\tau,
\] 
for every $t\geq 0$.
\end{theorem}
\begin{proof}
See \cite[Theorem 2]{Baleanu} and \cite{Kexue}.
\end{proof}

Now, we introduce basic properties of Mittag-Leffler function. These results are light refinements and adaption of the known results in the theory of Mittag-Leffler function to our case. To derive the estimations one uses the integral representation of Mittag-Leffler function, see e.g., Podlubny~\cite{Podlubny}. To save the length of the paper we do not give a full proof of the theorem, but give only sketch of the proof.
\begin{lemma}\label{lemma3}
Let $\lambda$ be an arbitrary complex number. There exist a positive real number  $m(\alpha,\lambda)$ such that for every $t\geq 1$ the following estimations hold:
\begin{itemize}
\item [(i)] if $\lambda\in\Lambda_\alpha^u$ then 
\begin{align*} 
\left|E_\alpha(\lambda t^{\alpha})-\frac{1}{\alpha}\exp{(\lambda^{\frac{1}{\alpha}}t)}\right|&\le \frac{m(\alpha,\lambda)}{t^{\alpha}},\\
\left|t^{\alpha-1}E_{\alpha,\alpha}(\lambda t^{\alpha})-\frac{1}{\alpha}\lambda^{\frac{1}{\alpha}-1
}\exp{(\lambda^{\frac{1}{\alpha}}t)}\right|&\le \frac{m(\alpha,\lambda)}{t^{\alpha+1}};
\end{align*}
\item [(ii)] if $\lambda\in\Lambda_\alpha^s$ then 
$$\left|t^{\alpha-1}E_{\alpha,\alpha}(\lambda t^{\alpha})\right|\le \frac{m(\alpha,\lambda)}{t^{\alpha+1}}.$$
\end{itemize}
\end{lemma}
For a proof of this theorem one uses integral representations of Mittag-Leffler functions and the method of estimation of the integrals similar to that of the proofs of Theorem 1.3 and  Theorem 1.4 in the book by Podlubny~\cite[pp. 32--34]{Podlubny}.
\begin{lemma}\label{lemma4} Let $\lambda\in\C\setminus\{0\}$. There exists a positive constant $K(\alpha,\lambda)$ such that for all $t \ge 0$ the following estimates hold: 
\begin{itemize}
\item [(i)] if $\lambda\in\Lambda_\alpha^u$ then
\begin{align*}
&\int_t^\infty \left|\lambda^{\frac{1}{\alpha}-1}E_\alpha(\lambda t^\alpha)
\exp(-\lambda^{\frac{1}{\alpha}}\tau)\right|\;d\tau
\leq K(\alpha,\lambda),\\
&\int_0^t\left|
\left((t-\tau)^{\alpha-1}E_{\alpha,\alpha}(\lambda(t-\tau)^{\alpha})- \lambda^{\frac{1}{\alpha}-1}E_\alpha(\lambda t^\alpha)\exp(-\lambda^{\frac{1}{\alpha}}\tau)\right)\right|\;d\tau\\
&\hspace{6.5cm} \leq K(\alpha,\lambda);
\end{align*}
\item [(ii)] if $\lambda\in\Lambda_\alpha^s$ then
\[
\int_0^t\left|(t-\tau)^{\alpha-1}E_{\alpha,\alpha}(\lambda(t-\tau)^{\alpha})\right|\;d\tau\leq K(\alpha,\lambda).
\]
\end{itemize}
\end{lemma}
\begin{proof}
The proof of this lemma follows easily by using Lemma \ref{lemma3} and repeating arguments used in the proof of Lemma 5 in \cite{Cong_2}.
\end{proof}

\begin{lemma}\label{LimitLemma}
For any function  $g\in C_\infty(\R_{\ge 0};\R)$ and $\lambda\in\Lambda_\alpha^u$, we have
\begin{align}\label{Eq4}
&\lim_{t\to\infty}\int_0^t
\notag (t-\tau)^{\alpha-1}\frac{E_{\alpha,\alpha}(\lambda(t-\tau)^\alpha)}{E_\alpha(\lambda t^\alpha)}g(\tau)\;d\tau\\
&\hspace*{5cm}=\lambda^{\frac{1}{\alpha}-1}\int_0^\infty\exp(-\lambda^{\frac{1}{\alpha}}\tau)g(\tau)\;d\tau.
\end{align}
\end{lemma}
\begin{proof}
Use Lemma \ref{lemma3}, Lemma \ref{lemma4} and  arguments analogous to those used in the proof of Lemma 8 in \cite{Cong_2}. 
\end{proof}

\section{Bounded solutions of inhomogeneous linear fractional differential equations}\label{sec.main}

\subsection{The scalar (complex-valued) case}
In this subsection, we consider the inhomogeneous scalar equation
\begin{equation}\label{ScalarEq}
\begin{split}
^{C\!}D^{\alpha}_{0+}x(t)&=\lambda x(t)+f(t),\\
x(0)&=x_0,
\end{split}
\end{equation}
where $\lambda\in \C\setminus\{0\}$, $x_0\in \C$ and $f\in C_{\infty}(\R_{\ge 0};\R)$. 
\begin{proposition}\label{Stable}
If $\lambda\in \Lambda_\alpha^s$ and $f\in C_\infty(\R_{\geq 0};\R)$, then all solutions of \eqref{ScalarEq} are bounded on $\R_{\ge 0}$. 

If, additionally, $f\in C^0_\infty(\R_{\geq 0};\R)$, then all solutions of this equation tend to $0$ as $t \to \infty$.
\end{proposition}
\begin{proof}
In the case $f \in C_{\infty}(\R_{\geq 0};\R)$, using the variation of constant formula provided by Theorem~\ref{Var_Const_Form} we see that the first assertion of this proposition follows from Lemma~\ref{lemma4}(ii) and the fact that $E_\alpha(\lambda t^\alpha)$ is bounded on $\R_{\geq 0}$. 

We now consider the case $f\in C_\infty^0(\R_{\geq 0};\R)$.
Let $\varepsilon>0$ be arbitrary. We can find a constant $T>0$ such that $|f(t)|<\varepsilon$ for all $t\ge T$. For $t>T+1$ and $x_0\in \C$, using Theorem~\ref{Var_Const_Form} we have the following formula for the solution of \eqref{ScalarEq} starting from $x_0$
\begin{multline*}
\varphi(t;0,x_0)=x_0 E_\alpha (\lambda t^\alpha)+\int_0^t (t-\tau)^{\alpha-1} E_{\alpha,\alpha}(\lambda(t-\tau)^\alpha)f(\tau)\;d\tau\\
\;\;\;=x_0 E_\alpha (\lambda t^\alpha)+\int_0^{T} (t-\tau)^{\alpha-1} E_{\alpha,\alpha}(\lambda(t-\tau)^\alpha)f(\tau)\;d\tau\\
+\int_{T}^{t-1} (t-\tau)^{\alpha-1} E_{\alpha,\alpha}(\lambda(t-\tau)^\alpha)f(\tau)\;d\tau+\int_{t-1}^t (t-\tau)^{\alpha-1} E_{\alpha,\alpha}(\lambda(t-\tau)^\alpha)f(\tau)\;d\tau.
\end{multline*} 
By virtue of Lemma \ref{lemma3}(ii), we have
\begin{equation}\label{ScalarEq1+}
\lim_{t\to \infty}x_0 E_\alpha(\lambda t^\alpha)=0.
\end{equation}
On the other hand, by a simple computation, we obtain
\begin{align}\label{ScalarEq2+}
\notag &\left|\int_T^{t-1}(t-\tau)^{\alpha-1}E_{\alpha,\alpha}(\lambda (t-\tau)^\alpha)f(\tau)\,d\tau\right| \le \varepsilon\,\int_{1}^{t-T}|\tau^{\alpha-1}E_{\alpha,\alpha}(\lambda \tau^\alpha)|\;d\tau\\
& \hspace*{1.0cm}\le \frac{\varepsilon\, m(\alpha,\lambda)}{\alpha} \qquad\hbox{(due to Lemma~\ref{lemma3}(ii))}
\end{align}
and 
\begin{align}\label{ScalarEq3+}
\notag  \left| \int_{t-1}^t (t-\tau)^{\alpha-1}E_{\alpha,\alpha}(\lambda (t-\tau)^\alpha)f(\tau)\;d\tau\right| \le \varepsilon\,\int_0^1|\tau^{\alpha-1}E_{\alpha,\alpha}(\lambda \tau^\alpha)|\;d\tau \\
\hspace*{1.0cm} \le \varepsilon\,  E_{\alpha,\alpha+1}(| \lambda |) \qquad
\hbox{(see \cite[formula (1.99), p.~24]{Podlubny}).} 
\end{align}
Furthermore, 
\begin{align}\label{ScalarEq4+}
\notag \Big| \int_0^T (t-\tau)^{\alpha-1} & E_{\alpha,\alpha}(\lambda (t-\tau)^\alpha) f(\tau) d\tau \Big| \\
\notag\hspace{1.0cm}& \le \sup_{t\in \R_{\geq 0}}|f(t)|\int_{t-T}^t |\tau^{\alpha-1} E_{\alpha,\alpha}(\lambda \tau^\alpha)|\,d\tau\\
\hspace{1.0cm}& \le \frac{m(\alpha,\lambda)\sup_{t\in\R_{\geq 0}}|f(t)|}{\alpha (t-T)^\alpha}\qquad\hbox{(due to Lemma~\ref{lemma3}(ii))}.
\end{align}
Since $\epsilon$ is arbitrarily, from  \eqref{ScalarEq1+}, \eqref{ScalarEq2+}, \eqref{ScalarEq3+}, \eqref{ScalarEq4+} we get
\[
\lim_{t\to \infty}|\varphi(t;0,x_0)|=0,
\]
which completes the proof. 
\end{proof}

\begin{proposition}\label{Unstable}
If 
$\lambda \in \Lambda_\alpha^u$ 
and $f\in C_\infty(\R_{\geq 0};\R)$, then equation \eqref{ScalarEq} has a unique bounded solution which is determined by the following formula
\begin{align*}
\varphi(t;0,\overline{x}_0) = \; &E_\alpha(\lambda t^\alpha)\Big(-\lambda^{\frac{1}{\alpha}-1}\int_0^\infty \exp{(-\lambda^{\frac{1}{\alpha}}\tau)}f(\tau)\;d\tau \Big)\\
&\hspace*{2cm}+\int_0^t (t-\tau)^{\alpha-1}E_{\alpha,\alpha}(\lambda (t-\tau)^\alpha)f(\tau)\;d\tau,
\end{align*}
where 
$$
\overline{x}_0:=-\lambda^{\frac{1}{\alpha}-1}\int_0^\infty \exp{(-\lambda^{\frac{1}{\alpha}}\tau)}f(\tau)\;d\tau.
$$
If, additionally, $f\in C_\infty^0(\R_{\geq 0};\R)$, then this solution $\varphi(0,\overline{x}_0;t)$ tends to 0 as $t$ tends to $\infty$.
\end{proposition}
\begin{proof}
If $f\in C_\infty(\R_{\geq 0};\R)$, by virtue of Lemma \ref{lemma4}(i), it is obvious that $\varphi(t;0,\overline{x}_0)$ is a bounded solution of \eqref{ScalarEq}. 

Now let us consider the case $f\in C_\infty^0(\R_{\geq 0};\R)$. Let $\varepsilon>0$ be an arbitrary positive real number. Then there exists a positive constant $T>0$ such that 
\begin{equation}\label{DecayCond}
|f(t)|\le \varepsilon \,\,\textup{for all}\,\, t\geq T.
\end{equation}
For any $t\geq T+ 1$ we put
\begin{align*}
I_1(t)&=E_\alpha(\lambda t^\alpha)\Big(-\lambda^{\frac{1}{\alpha}-1}\int_t^\infty \exp{(-\lambda^{\frac{1}{\alpha}}\tau)}f(\tau)\;d\tau\Big),\\
I_2(t)&=\int_0^T\left[(t-\tau)^{\alpha-1}E_{\alpha,\alpha}(\lambda(t-\tau)^\alpha)-\lambda^{\frac{1}{\alpha}-1}\exp{(-\lambda^{\frac{1}{\alpha}}\tau)}E_\alpha(\lambda t^\alpha)\right]f(\tau)\;d\tau,\\
I_3(t)&=\int_T^t\left[(t-\tau)^{\alpha-1}E_{\alpha,\alpha}(\lambda(t-\tau)^\alpha)-\lambda^{\frac{1}{\alpha}-1}\exp{(-\lambda^{\frac{1}{\alpha}}\tau)}E_\alpha(\lambda t^\alpha)\right]f(\tau)\;d\tau.
\end{align*}
By virtue of \eqref{DecayCond} and Lemma \ref{lemma4}(i), we have
\begin{align}\label{ScalarEq1-}
|I_1(t)|\le &\;  \varepsilon K(\alpha,\lambda).
\end{align}
Denote by $\varphi$ the argument of the complex number $\lambda$. Since $\lambda\in \Lambda^u_\alpha$ we have $0 \leq |\varphi| < \frac{\alpha\pi}{2}$, hence $\cos\frac{\varphi}{\alpha} > 0$. Due to Lemma~\ref{lemma3}(i), we have
\begin{align}\label{ScalarEq2-}
\notag |I_2(t)|&\le \sup_{t\in \R_{\geq 0}}|f(t)|\left[\int_{t-T}^t\frac{m(\alpha,\lambda)}{\tau^{\alpha+1}}\;d\tau+\frac{|\lambda^{\frac{1}{\alpha}-1}|\,m(\alpha,\lambda)}{t^\alpha r^{\frac{1}{\alpha}}\cos\frac{\varphi}{\alpha}}\right]\\
&\le \sup_{t\in \R_{\geq 0}}|f(t)|\left[\frac{m(\alpha,\lambda)}{\alpha\,(t-T)^\alpha}+\frac{m(\alpha,\lambda)}{r t^\alpha \,\cos\frac{\varphi}{\alpha}}\right].
\end{align}
Furthermore, by \eqref{DecayCond} and Lemma \ref{lemma4}(i), we have
\begin{equation}\label{ScalarEq3-}
|I_3(t)|\le \varepsilon K(\alpha,\lambda).
\end{equation}
From \eqref{ScalarEq1-}, \eqref{ScalarEq2-}, \eqref{ScalarEq3-} and the fact that $\varepsilon$ can be made arbitrarily small, it implies that
\[
\lim_{t\to \infty}\varphi(t;0,\overline{x}_0)=0.
\]
To complete the proof, it remains to show that Equation \eqref{ScalarEq} has exactly one bounded solution determined in the formulation of the Proposition. Indeed, assume that $\hat\varphi$ is another bounded solution of \eqref{ScalarEq}. Then the difference between two solution $\varphi-\hat\varphi$ is bounded. Furthermore, due to linearity of \eqref{ScalarEq} this difference is a solution of the following homogeneous equation
\[
^{C\!}D^\alpha_{0+}x(t)=\lambda x(t).
\] 
However, since $\lambda\in\Lambda_\alpha^u$ the only bounded solution of this equation is the trivial solution. Therefore, $\varphi-\hat\varphi =0$. The proof is complete.
\end{proof}
\begin{remark}
If $\lambda\in\R$ then all the discussions in Lemmas \ref{Stable} and \ref{Unstable} above can be carried out exclusively in the field of real numbers.
\end{remark}

\subsection{The high dimensional case}
For a matrix $A\in \R^{d\times d}$, let $\{\hat\lambda_1,\ldots,\hat\lambda_m\}$ be the collection of  all the distinct complex eigenvalues of $A$. By definition the spectrum of $A$ is $\sigma(A) := \{\hat\lambda_1,\ldots,\hat\lambda_m\}$. 
Consider the high dimensional inhomogeneous equation
\begin{equation}\label{GeneralEq}
^{C\!}D^\alpha_{0+}x(t)=Ax(t)+f(t),
\end{equation}
where $A\in \R^{d\times d}$, $x : \R_{\geq 0} \rightarrow \R^d$ and $f: \R_{\geq 0} \rightarrow \R^d$ is a continuous vector-valued function whose components belong to the space $C_\infty(\R_{\geq 0};\R)$. 

Let $T\in\C^{d\times d}$ be a nonsingular matrix  transforming $A$ into its Jordan normal form, i.e., 
\[
T^{-1}A T=\hbox{diag}(A_1,\dots,A_n),
\]
where for $i=1,\dots,n$ the block $A_i$ is of the following form
\[
A_i=\lambda_i\, \id_{d_i\times d_i}+\eta_i\, N_{d_i\times d_i},
\]
where $\eta_i\in\{0,1\}$, $\lambda_i \in \sigma(A)$, and the nilpotent matrix $N_{d_i\times d_i}$ is given by
\[
N_{d_i\times d_i}:=
\left(
      \begin{array}{*7{c}}
      0  &     1         &    0      & \cdots        &  0        \\
        0        & 0    &    1     &   \cdots      &              0\\
        \vdots &\vdots        &  \ddots         &          \ddots &\vdots\\
        0 &    0           &\cdots           &  0 &          1 \\

        0& 0  &\cdots                                          &0         & 0 \\
      \end{array}
    \right)_{d_i \times d_i}.
\]
Let us notice that by this transformation we go from the field of real numbers  out to the field of complex numbers, and we may remain in the field of real numbers only if all eigenvalues of $A$ are real. For a general real-valued matrix $A$ we may simply embed $\R$ into $\C$, consider $A$ as a complex-valued matrix and thus get the above Jordan form for $A$. Alternatively, we may use a more cumbersome real-valued Jordan form  (see 
Lancaster and Tismenetsky~\cite[Chapter 6, p. 243]{Lancaster}; for discussion on similar issue for FDE see also Diethelm~\cite[pp. 152--153]{Kai}). We note that the embedding of $\R$ into $\C$ preserves the norm of vectors, hence the embedding and returning back from $\C$ to $\R$ do not change the boundedness and decaying properties of the functions. Therefore, if by using embedding method we can show the boundedness or decaying property of solutions of real valued equation \eqref{GeneralEq} then the boundedness and decaying property of solutions are shown and valid for the {\em real-valued solutions} of \eqref{GeneralEq}.
For simplicity we use the embedding method and omit the detailed discussion on how to return back to the field of real numbers. 
Note also that such kind of technique is well known in the theory of ordinary differential equations. 

By the transformation $T$ we reduce \eqref{GeneralEq} to the Jordan case. Next we investigate the case of one Jordan block. Namely, for an arbitrary complex number $\lambda\in\C\setminus\{0\}$, we introduce a notation
\[
A_\lambda=\left(
      \begin{array}{*7{c}}
       \lambda   &     1         &    0      & \cdots        &  0        \\
        0        & \lambda    &    1      &   \cdots      &              0\\
        \vdots &\vdots        &  \ddots         &          \ddots &\vdots\\
        0 &    0           &\cdots           &  \lambda&          1 \\
                                       
        0& 0  &\cdots                                          &0         &\lambda \\
      \end{array}
    \right)_{d_\lambda\times d_\lambda}.
\]
Let $g_{\lambda}=(g_1^\lambda,\cdots,g_{d_\lambda}^\lambda)^{\rT}:\R_{\ge 0}\to \R^{d_\lambda}$ be a continuous vector-valued function and whose components belong to the space $C_\infty(\R_{\geq 0};\R)$. Let us consider the equation
\begin{equation}\label{Eq5}
^{C\!}D^{\alpha}_{0+}x(t)=A_\lambda x(t)+g^\lambda(t).
\end{equation}
This equation can be rewritten in the following form
\begin{eqnarray}
^{C\!}D^{\alpha}_0 x_1(t)&=& \lambda x_1(t)+x_2(t)+g^{\lambda}_1(t)\\
^{C\!}D^{\alpha}_0 x_2(t)&=& \lambda x_2(t)+x_3(t)+g^{\lambda}_2(t)\\
\notag & \ldots\\
\label{tam3} ^{C\!}D^{\alpha}_0 x_{d_\lambda-1}(t)&=&\lambda x_{d_\lambda-1}(t)+x_{d_\lambda}(t)+g^{\lambda}_{d_\lambda-1}(t)\\
\label{tam1} ^{C\!}D^{\alpha}_0 x_{d_\lambda}(t)&=&\lambda x_{d_\lambda}(t)+g^{\lambda}_{d_\lambda}(t).
\end{eqnarray}
\begin{proposition}\label{propo1}
Let $\lambda\in \Lambda_\alpha^s$ and assume that $g^\lambda \in C_\infty(\R_{\geq 0};\R^{d_\lambda})$. Then all solutions of \eqref{Eq5} are bounded. 

If, additionally, $g_\lambda \in C_\infty^0(\R_{\geq 0};\R^{d_\lambda})$ then all the solutions of \eqref{Eq5} tend to 0 as $t$ tends to $\infty$.
\end{proposition}
\begin{proof}
Assume that $g^\lambda \in C_\infty(\R_{\geq 0};\R^{d_\lambda})$.
Let $x^0=(x^0_1,\cdots,x^0_{d_\lambda})^{\rT}\in\R^{d_\lambda}$ be an arbitrary vector and  $\varphi(t;0,x^0)=(\varphi_1(t),\cdots,\varphi_{d_\lambda}(t))^{\rm T}$ denote the solution of \eqref{Eq5} satisfying the initial condition $\varphi(t;0, x^0)=x^0$. From \eqref{tam1}, we have
\begin{equation}\label{tam2}
\varphi_{d_\lambda}(t)=E_\alpha(\lambda t^\alpha)x^0_{d_\lambda}+\int_0^t (t-s)^{\alpha-1}E_{\alpha,\alpha}(\lambda (t-s)^\alpha)g^{\lambda}_{d_\lambda}(s)\;ds.
\end{equation}
It follows from Proposition~\ref{Stable} that $\varphi_{d_\lambda}$ is bounded in $\R_{\ge 0}$.
Substitute $\varphi_{d_\lambda}$ into \eqref{tam3} and applying Proposition~\ref{Stable} again
we get that $\varphi_{d_\lambda-1}$ is also bounded. Continue this process we will get that 
$\varphi_{d_\lambda-2},\cdots, \varphi_1$ are all bounded.

The case $g^\lambda \in C_\infty^0(\R_{\geq 0};\R^{d_\lambda})$ can be easily treated similarly.
\end{proof}

\begin{proposition}\label{propo2}
Let $\lambda\in \Lambda_\alpha^u$ and assume that $g^\lambda \in C_\infty(\R_{\geq 0};\R^{d_\lambda})$. Then, the equation \eqref{Eq5} has a unique bounded solution. 

If, additionally,   $g^\lambda \in C_\infty^0(\R_{\geq 0};\R^{d_\lambda})$ then this bounded solution 
tends to 0 as $t$ tends to $\infty$. 
\end{proposition}
\begin{proof}
Use arguments similar to that of the proof of Proposition~\ref{propo1} and with the application of  Proposition~\ref{Unstable}.
\end{proof}
\begin{remark}
If $\lambda\in\R$ then all the discussions in Lemmas \ref{propo1} and \ref{propo2} above can be carried out exclusively in the field of real numbers.
\end{remark}

Now we come to the investigation of the general case of \eqref{GeneralEq}. 
Denote by $\mathfrak{B}(A,f)$ the set of all bounded solutions of \eqref{GeneralEq}. We give a complete description the set  $\mathfrak{B}(A,f)$ in case the spectrum of $A$ satisfies the hyperbolicity condition \eqref{SpecCond}. 
Recall that by the transformation $T$ we reduce  $A$ to the Jordan form. Let us make the change of variable $x(\cdot) = Ty(\cdot)$ then \eqref{GeneralEq} is transformed into the equation
\begin{equation}\label{eqn.jordan}
^{C\!}D^{\alpha}_{0+}y(t) = By(t)+g(t),
\end{equation}
where $B$ is the Jordan normal form of $A$,
\begin{equation}\label{eqn.jordan1}
B= T^{-1}AT = \hbox{diag}(A_1,\dots,A_n), \quad \hbox{and} \quad g(t) = T^{-1} f(t).
\end{equation}
Note that the set $\mathfrak{B}(A,f)$ of all bounded solutions of \eqref{GeneralEq} can be found from the set $\mathfrak{B}(B,g)$ of all bounded solutions of \eqref{eqn.jordan} by the formula
\begin{equation}\label{eqn.jordan2}
\mathfrak{B}(A,f) = T\mathfrak{B}(B,g).
\end{equation}
Therefore, for description of the set of all bounded solutions of the general equation \eqref{GeneralEq} it suffices to do it in the Jordan case, i.e., for the equation \eqref{eqn.jordan}. Moreover, since $\sigma(B)=\sigma(A)$  if $A$ satisfies the hyperbolicity condition \eqref{SpecCond}  then $B$ satisfies the hyperbolicity condition $\sigma(B) \subset \Lambda_{\alpha}^u \cup \Lambda_{\alpha}^s$.

Assume that the spectrum of the Jordan matrix $B$ satisfies the hyperbolicity condition $\sigma(B) \subset \Lambda_{\alpha}^u \cup \Lambda_{\alpha}^s$, then without loss of generality we may rewrite \eqref{eqn.jordan} into the form
\begin{equation}\label{eq.tam}
^{C\!}D^{\alpha}_{0+}y(t) = \hbox{diag}(B^s,B^u)y(t)+(g^s(t),g^u(t))^{\rm T}
\end{equation}
where $B^{s/u}$ is the part of $B$ corresponding to the collection of all blocks with the eigenvalues belonging to $\Lambda_\alpha^{s/u}$. Now assume that $g^\lambda \in C_\infty(\R_{\geq 0};\R^{d_\lambda})$. Applying Proposition~\ref{propo2} to each Jordan block from $B^u$ we find a unique bounded solution $\bar\varphi(\cdot;0,\bar{y}^u)$ of the  equation (the unstable part of the equation \eqref{eq.tam})
$$
^{C\!}D^{\alpha}_{0+}y^u(t) = B^u y^u(t)+g^u(t);
$$
and applying Proposition~\ref{propo1} to each Jordan block from $B^s$ we find that for any initial value $y_0^s$ the solution $\varphi(\cdot;0,y_0^s)$ of the  equation (the stable part of the equation \eqref{eq.tam})
$$
^{C\!}D^{\alpha}_{0+}y^s(t) = B^s y^s(t)+g^s(t),\qquad y^s(0) = y_0^s,
$$
is bounded. The case the components of $g$ belong to $C^0_\infty(\R_{\geq 0};\R^{d_\lambda})$ can be treated similarly. Thus, we arrive at a theorem about the structure of the bounded (decay) solutions of fractional differential equations as follows.
\begin{theorem}[Structure of the bounded (decay) solutions of FDEs]\label{Main Result1}
Assume that the spectrum of $B$ satisfies the hyperbolicity condition $\sigma(B) \subset \Lambda_{\alpha}^u \cup \Lambda_{\alpha}^s$ and $g: [0,\infty)\to \R^d$ is a continuous vector-valued function whose components belong to the space $C_\infty(\R_{\geq 0};\R)$. Then the set of all bounded solutions of \eqref{eq.tam} is 
\[
\mathfrak{B}(B,g)=\left\{(\varphi(0,y_0^s;t),\bar\varphi(0,\bar{y}^u;t))^{\rm T} \right\},
\]
where the functions $\varphi(\cdot;0,y_0^s)$ and $\bar\varphi(\cdot;0,\bar{y}^u)$ are described in the paragraph preceding the formulation of the theorem. 

If, additionally, $g\in C_\infty^0(\R_{\geq 0};\R^d)$, then all the bounded solutions of \eqref{eq.tam} tend to 0 as $t$ tends to $\infty$.
\end{theorem}

Before going to formulation and proof of the main theorem of the paper about a necessary and sufficient condition for the inhomogeneous system \eqref{mainEq} to have at least one bounded solution, we give here two examples showing the existence of unbounded solutions of FDEs in nonhyperbolic case.

\begin{example}[Trivial linear part]\label{Triviallinearpart}
Consider the scalar fractional differential equation with trivial linear part
\[
^{C\!}D^{\alpha}_{0+}x(t)=\Gamma(1+\alpha).
\]
This equation is easily solved. Its general solution is $x_0 + t^\alpha$, and clearly no solution is bounded.
\end{example}

\begin{example}[Non hyperbolic linear part]\label{Nonhyperboliclinearpart}
Consider the scalar fractional differential equation
\begin{equation}\label{eq.ex}
^{C\!}D^{\alpha}_{0+}x(t)=\lambda x(t)+\exp{(ir^{\frac{1}{\alpha}}t)},
\end{equation}
where $\lambda=r(\cos\frac{\pi \alpha}{2} + i\sin\frac{\pi \alpha}{2})$. Let $x_0\in\C$ be arbitrary. By Theorem \ref{Var_Const_Form}, 
the solution $\varphi(\cdot;0,x_0)$ of \eqref{eq.ex} starting from $x_0$  satisfies
\[
\varphi(t;0,x_0)=E_\alpha(\lambda t^\alpha)\,x_0+\int_0^t (t-\tau)^{\alpha-1}E_{\alpha,\alpha}(\lambda (t-\tau)^\alpha)\exp{(ir^{\frac{1}{\alpha}}\tau)}\;d\tau.
\]
We claim that this solution is unbounded. Indeed, the quantity $x_0E_\alpha(\lambda t^\alpha)$ is bounded due to Podlubny~\cite[Theorem 1.1, p.~30]{Podlubny}, while the quantity 
$$
\int_{t-1}^t (t-\tau)^{\alpha-1}E_{\alpha,\alpha}(\lambda (t-\tau)^\alpha)\;d\tau
$$ 
is bounded by the following estimate
$$
\Big|\int_{t-1}^t (t-\tau)^{\alpha-1}E_{\alpha,\alpha}(\lambda (t-\tau)^\alpha)\;d\tau\Big| \le \int_0^1  s^{\alpha-1}|E_{\alpha,\alpha}(\lambda s^\alpha)|\;ds.
$$
Furthermore, by using Podlubny~\cite[Theorem 1.1, p.~30]{Podlubny}, it is not difficult to show that the quantity 
\[
\int_{0}^{t-1} (t-\tau)^{\alpha-1}E_{\alpha,\alpha}(\lambda (t-\tau)^\alpha)\exp{(ir^{\frac{1}{\alpha}}\tau)}\;d\tau
\]
is unbounded. Indeed, for $\theta \in (\frac{\alpha \pi}{2},\alpha \pi)$ is arbitrary but fixed and $\varepsilon \in (0,\frac{|\lambda|}{2})$ satisfies
\begin{equation}\label{unbouded_est}
|\lambda| t^\alpha-\varepsilon\geq |\lambda| t^\alpha \sin (\theta-\frac{\alpha \pi}{2})
\end{equation}
for all $t\geq 1$, we denote by $\gamma(\varepsilon,\theta)$ the contour consisting of the following three parts
\begin{itemize}
\item [(i)] $\text{arg}(z)=-\theta$, $|z|\ge \varepsilon$;
\item [(ii)] $-\theta\le \text{arg}(z)\le \theta$, $|z|=\varepsilon$;
\item [(iii)] $\text{arg}(z)=\theta$, $|z|\ge \varepsilon$.
\end{itemize}
The contour $\gamma(\varepsilon,\theta)$ divides the complex plane $(z)$ into two domains, which we denote by $G^{-}(\varepsilon,\theta)$ and $G^{+}(\varepsilon,\theta)$. These domains lie correspondingly on the left and on the right side of the contour $\gamma(\varepsilon,\theta)$. According to \cite[Theorem 1.1, p.~30]{Podlubny}, we have
\begin{align*}
& \int_{0}^{t-1} (t-\tau)^{\alpha-1}E_{\alpha,\alpha}(\lambda (t-\tau)^\alpha)\exp{(ir^{\frac{1}{\alpha}}\tau)}\;d\tau\\
&=\int_0^{t-1} (t-\tau)^{\alpha-1}\frac{1}{\alpha}\lambda^{\frac{1-\alpha}{\alpha}}(t-\tau)^{1-\alpha}\exp(\lambda^{\frac{1}{\alpha}}(t-\tau))\exp(i r^{\frac{1}{\alpha}}\tau)\;d\tau\\
&\hspace*{2.0cm}+\int_0^{t-1} (t-\tau)^{\alpha-1}\frac{1}{2\alpha \pi i}\int_{\gamma(\varepsilon,\theta)}\frac{\exp(\xi^{\frac{1}{\alpha}})\xi^{\frac{1-\alpha}{\alpha}}}{\xi-\lambda (t-\tau)^\alpha}\;d\xi \exp(i r^{\frac{1}{\alpha}}\tau)\;d\tau\\
&=I_4(t)+I_5(t).
\end{align*}
Clearly, we see that
\begin{align}
\notag I_4(t)&=\int_0^{t-1} (t-\tau)^{\alpha-1}\frac{1}{\alpha}\lambda^{\frac{1-\alpha}{\alpha}}(t-\tau)^{1-\alpha}\exp(\lambda^{\frac{1}{\alpha}}(t-\tau))\exp(i r^{\frac{1}{\alpha}}\tau)\;d\tau\\
&=\frac{\lambda^{\frac{1-\alpha}{\alpha}}}{\alpha}(t-1)\exp(\lambda^{\frac{1}{\alpha}}t).\label{unbouded_est1}
\end{align}
On the other hand, due to \eqref{unbouded_est}, we obtain
\begin{align}
|I_5(t)|&\leq \frac{\int_{\gamma(\varepsilon,\theta)}|\exp(\xi^{\frac{1}{\alpha}})\xi^{\frac{1-\alpha}{\alpha}}|\;d\xi}{2\alpha \pi \sin (\theta-\frac{\alpha \pi}{2})}\int_0^{t-1}\frac{(t-\tau)^{\alpha-1}}{|\lambda|(t-\tau)^\alpha}\;d\tau\\
&\leq \frac{\int_{\gamma(\varepsilon,\theta)}|\exp(\xi^{\frac{1}{\alpha}})\xi^{\frac{1-\alpha}{\alpha}}|\;d\xi}{2\alpha \pi |\lambda|\sin (\theta-\frac{\alpha \pi}{2})}\log t.\label{unbounded_est2}
\end{align}
From \eqref{unbouded_est1} and \eqref{unbounded_est2}, this implies that the quantity 
\[
\int_{0}^{t-1} (t-\tau)^{\alpha-1}E_{\alpha,\alpha}(\lambda (t-\tau)^\alpha)\exp{(ir^{\frac{1}{\alpha}}\tau)}\;d\tau
\]
is unbounded.
\end{example}

 \begin{theorem}[A Perron-type theorem for FDSs]\label{thm.main}
The inhomogeneous system \eqref{mainEq}
$$
^{C\!}D^\alpha_{0+}x(t)=Ax(t)+f(t)
$$
has at least one bounded solution for every $f\in C_\infty(\R_{\geq 0};\R^d)$ if only if the matrix $A$ 
satisfies the hyperbolicity condition \eqref{SpecCond}:
$$
\sigma(A) \subset \Lambda_{\alpha}^u \cup \Lambda_{\alpha}^s,
$$
where $\Lambda_{\alpha}^u, \Lambda_{\alpha}^s$ are defined by \eqref{SectorUnstable}--\eqref{SectorStable}.
\end{theorem}
 
\begin{proof}
Suppose that $A$ satisfies the hyperbolicity condition \eqref{SpecCond}, then by Theorem \ref{Main Result1} and \eqref{eqn.jordan1}--\eqref{eqn.jordan2} for any bounded continuous $f(\cdot)$ the FDE
\eqref{mainEq} has at least one bounded solution.

Now, assume that $A$ does not satisfies the hyperbolicity condition \eqref{SpecCond}, then there exists $\lambda \in \sigma (A)$ such that $\lambda=0$ or $\arg(\lambda)=\pm\frac{\alpha \pi}{2}$.
Without loss of generality, (transform $A$ to the Jordan form and change the order of coordinates if necessary) we can write $A$ in the form
\[
A=\hbox{diag}(A_\lambda,\lambda ).
\]
Choosing 
\[
f(t)=\begin{cases}
(f_\lambda (t),\exp{(ir^{\frac{1}{\alpha}}t)})^{\rm T},\quad &\text{if} \quad \arg(\lambda)=\pm\frac{\alpha \pi}{2}\\
(f_\lambda (t),\Gamma(1+\alpha))^{\rm T}, \quad & \text{if} \quad \lambda=0,
\end{cases}
\]
where $r=|\lambda|$ and $f_\lambda:\R_{\geq 0}\rightarrow \R^{d-1}$ is a bounded continuous function. In the case $\arg(\lambda)=\frac{\alpha \pi}{2}$, from \eqref{GeneralEq}, we see that the $d$-component of the solution $\varphi(\cdot;0,x^0)$ is
\[
\varphi_d(t)=E_\alpha(\lambda t^\alpha)\,x_d^0+\int_0^t (t-\tau)^{\alpha-1}E_{\alpha,\alpha}(\lambda (t-\tau)^\alpha)\exp{(ir^{\frac{1}{\alpha}}\tau)}\;d\tau.
\]
Due to Example \ref{Nonhyperboliclinearpart}, $\varphi_d(\cdot)$ is unbounded in $\R_{\geq 0}$ for any $x_d^0\in \R$, and thus the solution $\varphi(\cdot;0,x^0)$ is unbounded for any $x^0\in \R^d$. 
 Similarly, in case
$\arg(\lambda)=-\frac{\alpha \pi}{2}$ any solution of \eqref{mainEq} is unbounded. 
Now, if $\lambda=0$, due to Example \ref{Triviallinearpart}, the $d$-component of the solution $\varphi(\cdot;0,x^0)$ of \eqref{GeneralEq} is
\[
x_d^0 + t^\alpha.
\]
This shows that $\varphi(\cdot;0,x^0)$ is unbounded for any initial condition $x^0$. 

Thus, we have shown that if  $A$ does not satisfies the hyperbolicity condition \eqref{SpecCond}, then any solution of \eqref{mainEq} is unbounded.
The proof is complete.
\end{proof}

\section*{Acknowledgement}
The work of authors is supported by the Vietnam National Foundation for Science and Technology Development (NAFOSTED) under Grant Number 101.03-2014.42.

\end{document}